\documentclass[11pt]{amsart}

\usepackage{amsmath, url}
\usepackage{amssymb}
\usepackage{graphicx}

\newtheorem{theorem}{Theorem}[section]
\newtheorem{corollary}[theorem]{Corollary}
\newtheorem{lemma}[theorem]{Lemma}

\theoremstyle{definition}
\newtheorem{definition}[theorem]{Definition}
\newtheorem{example}[theorem]{Example}
\newtheorem{remark}[theorem]{Remark}

\numberwithin{equation}{section}

\def\beq{\begin{equation}}
\def\eeq{\end{equation}}

\def\bN{\mathbf{N}}

\def\bP{\mathbf{P}}

\def\bR{\mathbf{R}}

\def\bZ{\mathbf{Z}}

\def\bf{\textbf{}}

\def\bm{\mathbf{m}}

\def\b1{{\boldsymbol{1}}}

\def\cA{\mathcal{A}}
\def\cB{\mathcal{B}}

\def\cD{\mathcal{D}}

\def\cL{\mathcal{L}}
\def\cM{\mathcal{M}}

\def\cP{\mathcal{P}}

\def\cR{\mathcal{R}}

\def\beq{\begin{equation}}
\def\eeq{\end{equation}}
\def\beqn{\begin{equation*}}
\def\eeqn{\end{equation*}}
\def\beqy{\begin{eqnarray}}
\def\eeqy{\end{eqnarray}}
\def\beqyn{\begin{eqnarray*}}
\def\eeqyn{\end{eqnarray*}}

\def\bN{\mathbf{N}}

\def\bP{\mathbf{P}}

\def\bR{\mathbf{R}}

\def\bZ{\mathbf{Z}}

\def\bm{\mathbf{m}}

\def\b1{{\boldsymbol{1}}}

\def\cA{\mathcal{A}}
\def\cB{\mathcal{B}}

\def\cD{\mathcal{D}}

\def\cL{\mathcal{L}}
\def\cM{\mathcal{M}}

\def\cP{\mathcal{P}}

\def\cR{\mathcal{R}}

\newcommand{\calA}{\mathcal{A}}

\newcommand{\calC}{\mathcal{C}}
\newcommand{\calD}{\mathcal{D}}

\newcommand{\calF}{\mathcal{F}}

\newcommand{\calL}{\mathcal{L}}
\newcommand{\calM}{\mathcal{M}}

\newcommand{\norm}[1]{\left\lVert#1\right\rVert}
\newcommand{\trace}{\text{trace}}
\newcommand{\bbP}{\mathbb{P}^1}
\newcommand{\hO}{\hat{\Omega}}
\newcommand{\htheta}{\hat{\theta}}
\newcommand{\GL}{\text{GL}(2,\bR)}

\newcommand{\uu}{\underline{u}}
\newcommand{\uU}{\underline{U}}
\newcommand{\SL}{\text{SL}}

\title{Hyperbolicity and Ergodicity in Randomly Perturbed Billiard Systems}

\author[K. Nguyen]{Kien Nguyen}

\address[K. Nguyen]{University of Massachusetts, Amherst}
\email{{\tt kiennguyen@umass.edu}}

\author[H-K. Zhang]{HongKun Zhang}
\address[H-K. Zhang]{University of Massachusetts Amherst \& Great Bay University}
\email{\tt hongkunz@umass.edu}



\begin{document}

\begin{abstract}
In this paper, we investigate the Lyapunov exponents and hyperbolicity of random billiard systems under small stochastic perturbations. Our study begins with a review of the necessary theoretical background on Lyapunov exponents for stationary sequences of matrices and Markov processes. We then consider several classical billiard systems and demonstrate that random perturbations can lead to significant changes in their dynamical properties. Specifically, we show that non-circular elliptic and lemon billiards exhibit both ergodicity and hyperbolicity under random perturbations, whereas circular billiards, while ergodic, maintain zero Lyapunov exponents. We also establish conditions under which the largest Lyapunov exponent becomes positive, indicating the emergence of hyperbolic behavior in these systems. These results provide new insights into the stability and chaotic properties of dynamical systems subjected to random noise, with implications for the broader study of smooth dynamical systems.
\end{abstract}

\maketitle


\section{Introduction}

One of the central problems in smooth dynamics is establishing the hyperbolicity of a given system. Hyperbolic systems are characterized by their sensitive dependence on initial conditions, and a key tool in their analysis is the study of Lyapunov exponents. These exponents provide crucial information regarding the stability of the dynamics in response to small perturbations in initial conditions.

Let \( M \) denote the collision space of a billiard, and let \( F: M \to M \) represent the billiard map on \( M \). The map \( F \) is a diffeomorphism on an open, dense subset of \( M \) and preserves a natural probability measure \( \mu \) on \( M \). By Oseledets' Theorem \cite{cb}, assuming integrability conditions and boundedness of the boundary curvature, the Lyapunov exponents \( \lambda_1(x) \geq \lambda_2(x) \) exist for \( \mu \)-almost every point \( x \in M \).

A point \( x \) is called \emph{hyperbolic} if its Lyapunov exponents are nonzero, and the map \( F \) is termed hyperbolic if \( \mu \)-almost every point in \( M \) is hyperbolic. For the class of billiards considered here, the two Lyapunov exponents are of opposite sign \cite{cb}, and a point is hyperbolic if the largest exponent \( \lambda_1(x) \) is positive, indicating strong expansion in one direction (and strong contraction in another) at that point.

The standard approach to proving the positivity of a Lyapunov exponent involves establishing the existence of a strictly invariant cone field on the tangent space \cite{don91}. However, this method can be challenging to apply to many billiard systems, such as the moon billiards discussed in \cite{zm}.

A different approach involves introducing random perturbations to the billiard system and studying the resulting stochastic dynamics. Although this method does not directly solve the deterministic problem, it provides valuable insights into the system's behavior under small random perturbations. Several works have explored stochastic perturbations of billiards, including \cite{magnetic}, \cite{zfspectrum}, \cite{zfgaps}, \cite{zfn},  \cite{markarian}, \cite{zsy}, \cite{zdp}. However, in these studies, the invariant measure of the perturbed system does not coincide with the natural measure of the deterministic billiard map, and the positivity of the Lyapunov exponent is not established. Our work is closely related to \cite{markarian} by Markarian et al. and \cite{bxy} by Blumenthal, Xue, and Young. In \cite{bxy}, the authors consider a random perturbation that preserves the invariant measure of the unperturbed system. However, the system under consideration is the Chirikov standard map on the torus, with a different type of perturbation.

In this paper, we demonstrate that adding small noise to a system at each iteration can lead to the positivity of the largest Lyapunov exponent, provided there is an initial source of hyperbolicity. Even for systems with zero Lyapunov exponents on a set of full measure, the presence of noise can render the Lyapunov exponent positive, regardless of the noise's magnitude. On the other hand, we show that circular billiards cannot exhibit positive Lyapunov exponents, even under significant perturbations. This is because the noise is added independently of the points, so the derivative remains unperturbed, while the original circular billiard is linear with zero Lyapunov exponent. However, in the case of non-circular elliptic billiards, we observe positive Lyapunov exponents due to the presence of a hyperbolic periodic point, which acts as a source of hyperbolicity. We conjecture that circular billiards are the only smooth and convex billiards that do not become hyperbolic under perturbation.

The paper is organized as follows. In Section 2, we review the necessary background on Markov processes, which are essential as the perturbation applied to the billiard map results in a Markov transition function on the collision space. Consequently, in the perturbed system, each trajectory can be viewed as a realization of a Markov process. In Section 3, we define the Lyapunov exponent for a stationary sequence of matrices associated with a stochastic process and provide a necessary condition under which the Lyapunov exponents are zero. Section 4 is dedicated to a detailed exploration of the perturbations applied to various classical billiards. Here, we establish the ergodicity and hyperbolicity of the perturbed systems, showing that non-circular elliptic and lemon billiards exhibit both properties. Additionally, we demonstrate that while random circular billiards maintain ergodicity, they do not exhibit hyperbolicity, as their Lyapunov exponents remain zero.

\section{Preliminaries on Lyapunov exponents of stationary sequences of matrices}

In this section, we review some background information on the Lyapunov exponents of a stationary sequence of matrices. Most of the materials in this section can be found in \cite{led1} and \cite{led2}. Let $(\Omega, \cA, \bP)$ be a probability space and $\theta: (\Omega, \cA) \to (\Omega, \cA)$ a measurable map that preserves the probability measure $\bP$. We denote by $\GL$ the group of $2\times 2$ real invertible matrices. When viewed as a measurable space, the group $\GL$ is equipped with its Borel $\sigma$-algebra. Let $A: \Omega \to \GL$ be a measurable map. Let $A_n = A\circ \theta^n$. Then the sequence $(A_n)_{n\geq 0}$ is a stochastic process defined on the underlying probability space $(\Omega, \cA, \bP)$ with values in $\GL$. In fact, $(A_n)_{n\geq 0}$ is a stationary stochastic process because of the invariance of the measure $\bP$ under the map $\theta$. 

We construct another sequence $(A^{(n)})_{n\geq 1}$ of matrices by:
\begin{equation}
  A^{(n)}(\omega) := A_{n-1}(\omega)\cdot A_{n-2}(\omega) \cdots A(\omega)
\end{equation}
for any $n \geq 1$.


\begin{definition}\label{defmaxlya}
Let $(\Omega, \cA, \bP,\theta, A)$ be as above. The Lyapunov exponent at $\omega$ of the sequence $(A_{n}(\omega))_{n\geq 0}$ is defined to be:
\begin{equation}
\lambda(\omega) = \lim_{n\to \infty}\frac{1}{n}\log\norm{A^{(n)}(\omega)}
\end{equation}
if the limit exists. Here the norm $\|\cdot\|$ is the operator norm of a linear map $\bR^2\to \bR^2$.
\end{definition}
By the subadditivity of the sequence $(\|A^{(n)}(x)\|)_{n\geq 0}$ and stationarity of the sequence $(A_n)_n\geq 0$, we have the following lemma:

\begin{lemma}
For $\bP$-almost every $\omega$, the Lyapunov exponent at $\omega$ of the sequence $(A_{n})_{n\geq 0}$ exists in $\bR\cup \{-\infty\}$.
\end{lemma}

\begin{lemma}[\cite{led1} Proposition 1.1]\label{lyaintro}
  Let $(\Omega, \cA, \bP,\theta, A)$ be as defined above. Suppose that:
  \begin{equation}
    \int_{\Omega}\log^+ \norm{A(\omega)} \bP(d\omega) < \infty,
  \end{equation}
  where $\log^+ = \max(\log, 0)$. Then the following two limits
  \begin{equation}
    \lim_{n\to \infty}\frac{1}{n}\int_{\Omega}\log \norm{A^{(n)}(\omega)}\bP(d\omega)
  \end{equation}
  and 
  \begin{equation}
    \lim_{n\to \infty}\frac{1}{n}\int_{\Omega}\log |\det{A^{(n)}(\omega)}|\bP(d\omega)
  \end{equation}
  exist in the extended real line $\bR \cup \left\{- \infty \right\}$.
\end{lemma}

  Under the conditions of Lemma \ref{lyaintro}, let $\lambda_1$ and $\lambda_2$ be real numbers such that:

  \begin{equation}\label{1st}
    \lim_{n\to \infty}\frac{1}{n}\int_{\Omega}\log \norm{A^{(n)}(\omega)}\bP(d\omega) = \lambda_1
  \end{equation}
  and 
  \begin{equation}\label{2nd}
    \lim_{n\to \infty}\frac{1}{n}\int_{\Omega}\log |\det{A^{(n)}(\omega)}|\bP(d\omega) = \lambda_1 + \lambda_2.
  \end{equation}
  We set $\lambda_1$ or $\lambda_2$ to be $-\infty$ if the first or second limit is $-\infty$. We call the numbers $\lambda_1$ and $\lambda_2$ the {\it Lyapunov exponents} of the stationary process $(A_{n})_{n\geq 0}$.

We recall that any real square matrix $A$ can always be decomposed as $A = U\Sigma V^T$ where $U$ and $V^T$ are orthogonal matrices and
$\Sigma = 
\begin{pmatrix}
  \sigma_1(A)& 0\\
  0& \sigma_2(A)
\end{pmatrix}$ is a diagonal matrix with $\sigma_1(A) \geq \sigma_2(A) \geq 0$. This is called the Singular Value Decomposition of the matrix $A$. This decomposition tells us that geometrically a linear transformation is a composite of a rotation, a scaling and another rotation. The number $\sigma_1(A)$ is the larger scaling factor among $\sigma_1(A)$ and $\sigma_2(A)$. The columns of $V$ and $U$ tell us the directions in $\bR^2$ in which we will see the largest and the smallest scaling, and where those directions move to after the transformation. The following lemma allows us to write the Lyapunov exponents in terms of the scaling factors.

\begin{lemma}
  Under the same conditions as in Lemma \ref{lyaintro}, we have that:
  \begin{equation}
    \lambda_i = \lim_{n \to \infty}\frac{1}{n}\int_{\Omega}\log \sigma_i(A^{(n)}(\omega))\bP(d\omega)
  \end{equation}
  \label{lyasigma}
  for $i = 1, 2$. It is then clear that $\lambda_1 \geq \lambda_2$.
\end{lemma}

In fact, since we have that
\begin{equation*}
  \det A^{(n)}(\omega) := \det A_{n-1}(\omega)\cdot \det A_{n-2}(\omega) \cdots \det A(\omega)
  \end{equation*}
  and that the map $\theta:\Omega \to \Omega$ preserves the measure $\bP$, which implies that the sequence $(A_n)_{n\geq 0}$ is stationary,  we could drop a limit sign to have:
  \begin{align}\label{lyasumequaldet}
    \lambda_1 + \lambda_2 &= \lim_{n\to \infty}\frac{1}{n}\int_{\Omega}\log|\det A^{(n)}(\omega)|\bP(d\omega)\\
  &=\int_{\Omega}\log |\det(A(\omega))|\bP(d\omega)
\end{align}

The Lyapunov exponent $\lambda(\omega)$ at $\omega$ in Definition \ref{defmaxlya} of the sequence $(A_{n}(\omega))_{n\geq 0}$ can be viewed as the logarithm of the rate of expansion (or contraction) of vectors along the path starting from $\omega$. If the underlying dynamical system is ergodic, they are constant $\bP$-almost everywhere:

\begin{theorem}[\cite{led1} Theorem 2.6]\label{lyaergodic}
  Let $(\Omega, \cA, \bP, \theta)$ be an ergodic system, $A: \Omega \to \GL$ a measurable map such that:
  \begin{equation}
    \int_{\Omega}\log^+\norm{A(\omega)}\bP(d\omega) < \infty.
  \end{equation}
  Let $\lambda_1 \geq \lambda_2$ be the Lyapunov exponents of the stationary process $(A_{n})_{n\geq 0}$. Then we have:
  \begin{equation}
    \lambda_1 = \lim_{n\to \infty}\frac{1}{n}\log\norm{A^{(n)}(\omega)} \hspace{1cm} \bP-a.s.
  \end{equation}
  and 
  \begin{equation}
    \lambda_1 + \lambda_2 = \lim_{n\to \infty}\frac{1}{n}\log|\det{A^{(n)}(\omega)}| \hspace{1cm} \bP-a.s.
  \end{equation}
\end{theorem}
We have a straightforward corollary of the Theorem \ref{lyaergodic} in the case the function $A$ takes values in the set of matrices with determinant equals 1. Let $\SL(2,\bR)$ denote the set of $2\times 2$ real matrices with determinant 1.
\begin{corollary}\label{sumlya0}
  Let $(\Omega, \cA, \bP, \theta)$ be an ergodic system and $A:\Omega \to \SL(2,\bR)$ a measurable map such that
  \[
    \int_{\Omega}\log \max(\norm{A(\omega)}, \norm{A(\omega)^{-1}})\bP(d\omega) < \infty.
  \]
  Then both $\lambda_1$ and $\lambda_2$ are finite and moreover, $\lambda_1 + \lambda_2 = 0$.
\end{corollary}

We are interested in the necessary conditions to have $\lambda_1 = \lambda_2$. Later we will use these criteria to show that $\lambda_1 \neq \lambda_2$ for certain systems by way of contradiction. The idea behind these necessary conditions is that if the two Lyapunov exponents are equal then a very special condition on measurability must be satisfied. Avila and Viana in \cite{av} discussed this phenomenon in a more general setting.

Let $\bbP$ be the real projective space of dimension 1. Elements in $\bbP$ are equivalence classes of the vectors in $\bR^2$ where two nonzero vectors $v$ and $w$ are said to be equivalent if they are parallel. For any nonzero vector $v\in \bR^2$ we denote by $[v]$ its equivalent class. For any $\omega\in \Omega$, we have $A(\omega)$ is a matrix in $\GL$. The matrix $A(\omega)$ is a linear transformation on the vector space $\bR^2$ and hence induces a map on $\bbP$:
\begin{equation}\label{actiononP1}
A(\omega)([v]) = [A(\omega)(v)] \hspace{0.5cm}\text{ for any } [v]\in \bbP.
\end{equation}

Let $\hO = \Omega \times \mathbb{P}^1$ and define a map $\htheta : \Omega\times\bbP \to \Omega\times\bbP$ by:
\[
  \htheta(\omega, \hat{v}) = (\theta(\omega), A(\omega)(\hat{v})).
\]
Let $\pi_1: \hO \to \Omega$ be the projection map onto the first component. Any probability measure $\xi$ on $\hO$ such that ${\pi_1}_*\xi = \bP$ can be {\it disintegrated} into a family $\{\xi_{\omega}:\omega \in \Omega\}$ of probability measures on $\bbP$ such that the function $\omega \mapsto \xi_{\omega}$ is $\cA$-measurable. This family is essentially unique and each $\xi_{\omega}$ is supported on the fibre $p_1^{-1}(\{\omega\}) \cong \bbP$. We only consider measures $\xi$ that projects to $\bP$. 

We have the following theorem of Ledrappier:
\begin{theorem}[\cite{led2} Theorem 1]
  Let $(\Omega,\cA, \bP,\theta)$ be a measure-preserving dynamical system, not necessarily ergodic. Let $A: \Omega \to \GL$ be a measurable function such that
    \[
    \int_{\Omega}\log \max(\norm{A(\omega)}, \norm{A(\omega)^{-1}})\bP(d\omega) < \infty.
    \]
    Let $\cA_0 \subset \cA$ be a sub $\sigma$-algebra such that both $\theta$ and  $A$ are $\cA_0$-measurable and that $\cA$ can be generated by all the iterates $\theta^{n}(\cA_0)$ of $\cA_0$, $n\in \bZ$.

    Suppose that $\lambda_1 = \lambda_2$. Then any disintegration of a $\htheta$-invariant measure $\xi$ is $\cA_0$-measurable (modulo null sets).
  \label{led}
\end{theorem}

\section{Lyapunov exponents of a stationary sequence of matrices along a Markov process}

In this section, we first summarize some essential concepts and notations related to Markov processes that will be used throughout the paper. For a more detailed exposition, we refer the reader to standard texts such as \cite{walters}, \cite{williams}, \cite{mt}, \cite{durrett}, \cite{sinaiprob}.

Let $M$ be a complete separable metric space with Borel $\sigma$-algebra $\cB$. Let   $P:M\times\cB \to [0,1]$ be a Markov transition function on $M$.
The Ruelle transfer operator $\calL$ associated with $P$ is defined on the set $\cP(M)$ of probability measures on $M$ by
\begin{equation}
\cL\mu(B) = \int_M P(x,B)\mu(dx),
\end{equation}
for any $\mu \in \cP(M)$ and $B\in \cB$.

For any $x\in M$, let $\delta_x$ denote the Dirac measure at $x$. The one-step image of $\delta_x$ under $\calL$ is given by $\cL\delta_x(B) = P(x,B)$ for every $B\in \cB$. We can define the $n$-step transition function $P^n$ by
\begin{equation}
  P^n(x,B) = \cL^n\delta_x(B),
\end{equation}
for any $x\in M$ and $B\in \cB$. One can check that the $n$-step transition function satisfies the recursion
\begin{equation}
P^n(x,B) = \int_M P^{n-1}(y,B)P(x,dy)=  \int_M \cL^{n-1}\delta_y(B) \,P(x,dy),
\end{equation}
for any $x\in M$ and $B\in \cB$.

Given a transition function $P$ on $M$ and an initial probability measure $\mu_0$, we can construct a Markov process $\{X_n\}_{n\geq 0}$ on $M$ such that the distribution of $X_{n+k}$ given $X_n = x$ is given by $P^k(x,\cdot)$.
It is known that if $\{X_n\}_{n\geq 0}$ is a Markov process with transition function $P$, then
\[
  \bP(X_{k} \in B \mid X_0 = x) = P^k(x, B),	
\]
for any $k \geq 1$ and $B\in \cB$.

We say a probability measure $\mu \in \cP(M)$ is invariant with respect to $P$ if $\cL\mu = \mu$, i.e.,
\begin{equation}
	\mu(B) = \int_M P(x,B)\mu(dx),
\end{equation}
for all $B \in \cB$. Moreover, 
an invariant measure $\mu \in \cP(M)$ is called ergodic with respect to $P$ if, for any $A \in \cA$, $\theta^{-1}(A) = A$ implies $\bP_{\mu}(A)$ is either 0 or 1. The measure $\mu$ is mixing with respect to $P$ if for any $A, B \in \cA$,
\begin{equation}
\bP_{\mu}(\theta^{-n}A \cap B) \to \bP_{\mu}(A)\bP_{\mu}(B),
\end{equation}
as $n \to \infty$.

Consider a Markov process $(X_n)_{n\geq 0}$ with transition function $P$ and take values in $M$ with initial probability measure $\mu$ that is $\calL$-invariant. Let $(\Omega, \cA, \bP_{\mu})$ be the canonical underlying probability space for $(X_n)_{n\geq 0}$ constructed as in Section 2. Recall that an element $\omega \in \Omega$ is of the form $\omega = (x_0, x_1, \dots)$ and the random variables $X_n$'s are coordinate maps:
\[
  X_n(\omega) = x_n \text{ for } n \geq 0.
\]
The shift map $\theta$ on $\Omega$ is:
\[
  \theta(x_0, x_1, \dots) = (x_1, x_2, \dots).
\]
Since $\mu$ is invariant for $\cL$, the shift map $\theta$ preserves the measure $\bP_{\mu}$ and thus $(X_n)_{n\geq 0}$ is a stationary Markov process.

Let $A: M \to \SL(2,\bR)$ be a measurable map satisfying the condition:
\begin{flalign}
 {\mathbf{(H1)}}\hspace{0.5cm} \int_{M}\log \max(\norm{A(x)}, \norm{A(x)^{-1}}) \mu(dx) < \infty,
  \label{ca1}
\end{flalign}

With a slight abuse of notation, we define a function $A: \Omega \to \SL(2,\bR)$ by setting:
\[
  A(\omega) := A(x_0)
\]
for any $\omega = (x_0, x_1, \dots) \in \Omega$.

Let $\lambda_1 \geq\lambda_2$ be the Lyapunov exponents for the process $(A_n = A\circ \theta^n)_{n \geq0}$. By corollary \ref{sumlya0} we know that $\lambda_1 + \lambda_2 = 0$.

\begin{lemma}[\cite{led2} Corollary 2]\label{lambda1=0}


If  $\lambda_1 = 0$ then there exists a measurable family $\{\xi_x: x\in M\}$ of probability measures on the real projective line $\bbP$ such that for $\mu$-almost every $x\in M$:
\begin{equation}
	\xi_y = (A(x))_*\xi_x
\end{equation}
for $P(x, .)$-almost every $y$.
\end{lemma}

\section{Ergodicity and hyperbolicity of randomly perturbed billiards}

In this section, we consider a random perturbation to certain dynamical billiards and prove the ergodicity and hyperbolicity of the perturbed systems. Let us first recall some basic information about billiards. There are several introductory references to billiards including, but not limited to: \cite{cb}, \cite{kt}, \cite{tb}.

We call a billiard table $\cD$ {\it classical} if it satisfies the following assumption:

Assumption $\mathbf{(HB)}$: the interior $\cD_0$ is a compact and connected open domain in $\bR^2$, and the boundary $\partial \calD$ consists of finitely many piece-wise $C^3$ simple closed curves:
\begin{equation}
\partial \cD = \Gamma_1 \cup \Gamma_2 \dots \cup \Gamma_n,\hspace{0.5cm} \text{ , } n\geq 1.
\end{equation}
Each curve $\Gamma_i$ is given by a piecewise $C^3$ map $\gamma_i: [a_i, b_i] \to \bR^2$, which is injective on $[a_i, b_i)$ and satisfies $\gamma_i(a_i) = \gamma_i(b_i)$. Assume further that  $\gamma_i((a_i,b_i))$, $i = 1, \dots, n$, are disjoint.

Fix an orientation on each component $\Gamma_i$ so that the interior of the table lies on the left-hand side of $\Gamma_i$. We parametrised the $\Gamma_i$'s by their arclengths.

A point particle is moving inside $\cD$ and colliding with the boundary $\partial \cD$. Let $q(t)\in \cD$ be the position and $v(t)\in \bR^2$ the velocity of the particle at time $t\in \bR$. Between two collisions with the boundary, $q\in \cD_0$, the particle moves in the interior with constant velocity. At a collision with the smmooth part of the boundary, $q \in \partial\cD$, let $v^-$ and $v^+$ denote the pre-collisional and post-colllisional velocity vectors, respectively, and let $n$ be the unit normal vector to the boundary at $q$ pointing inward the table. Then we have:
\begin{equation}
v^+ = v^- -2(v^-,n)n.
\end{equation}

\begin{figure}[h]

\centering
\includegraphics[width=0.5\textwidth]{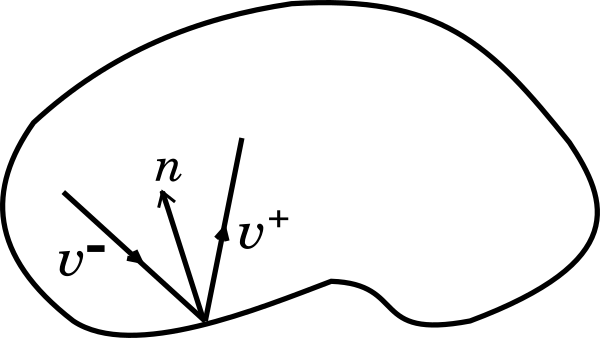}
\caption{Example of a collision in a billiard table}
\end{figure}

Let $\calM$ be the {\it collision space} of the billiard map on $\calD$. Every point $x\in \cM$ is a pair of its position $q$ and post-collisional velocity vector $v$. The boundary $\partial\calD$ is parametrised by the arc-length parameter $r$ in the chosen direction. For each point $x\in \cM$, the angle of reflection $\varphi$ is the directional angle from $v$ to the inward normal vector $n$. Note that $-\pi/2\leq \varphi\leq \pi/2$. Thus we have a coordinate system $r,\varphi$ on $\calM$.

For each $i = 1, 2,\dots,n$, let $\calM_i$ be the collision space for collisions that happen on $\Gamma_i$, then:
\begin{equation}
\cM = \cM_1\cup \cdots \cup \cM_n.
\end{equation}
For each $\Gamma_i$, since it is parametrised by arclength, we assume that it has length $|\Gamma_i| = b_i - a_i$. Let $\cR_i = [a_i, b_i]\times[-\pi/2, \pi/2]$. Then $\cM_i$ is a cylinder obtained by identify the two edges $\{r = a_i\}$ and $\{r = b_i\}$ of $\cR_i$ with each other.

Let $\calF:\calM \to \calM$ be the billiard map. It sends a point $(r,\varphi)\in \calM$ to $(r_1,\varphi_1)\in \calM$ at the next collision. Let $|\partial \cD|$ be the length of $\partial\calD$. By Lemma 2.35 in \cite{cb}, the collision map $\calF: \calM \to \calM$ preserves a probability measure $\mu$ on $\calM$ defined by:
\begin{equation}
    d\mu  = \frac{1}{2|\partial \cD| }\cos(\varphi)drd\varphi.  
\end{equation}

In this section, however, we will use the coordinate system given by $r$ and $s$, where $s = \sin(\varphi)$. For each $i = 1, 2, \dots, n$, let $R_i = [a_i, b_i] \times [-1,1]$ and $M_i$ the cylinder obtained by identify two edges $\{r = a_i\}$ and $\{r = b_i\}$ of the rectangle $R_i$ with each other. The collision space in this setting is $M = M_1\cup \cdots \cup M_n$ and the billiard map is now denoted by $F: M\to M$.

Let $S_1$ be the set of $x\in M$ such that the corresponding trajectory on the billiard table will hit the corners, or tangential to a dispersing wall. The billiard map $F$ is a $C^2$ diffeomorphism from $M\setminus S_1$ onto its image and $S_1$ is considered as the {\it singularity} of $F$.

Let $x=(r,\sin(\varphi))$ be any point in $M$ and $x_1 = F(x)=(r_1,\sin(\varphi_1))$ be the next collision, where $(r_1, \varphi_1) = \calF(r,\varphi)$. We denote by $K$ and $K_1$ the curvatures of the boundary at the collision points for $x$ and $x_1$, respectively, and by $\tau$ the distance of between those 2 collision points in the table. The differential of the billiard map, in the coordinates $r$ and $s= \sin(\varphi)$, is given by the formula:
\begin{equation}
DF(x) = 
	\begin{pmatrix}
		1 & 0\\
		0 & \cos(\varphi_1)\\
	\end{pmatrix}
\frac{-1}{\cos(\varphi_1)}\begin{pmatrix}
-\tau K + \cos(\varphi) & \tau\\
\tau KK_1 - K\cos(\varphi_1) - K_1\cos(\varphi)&-\tau K_1 +\cos(\varphi_1)
\end{pmatrix}
	\begin{pmatrix}
		1 & 0\\
		0 & \frac{1}{\cos(\varphi)}
	\end{pmatrix}.
\end{equation}

Note that $\det(DF(x)) = 1$ and thus the billiard map preserves the multiple of the Lebesgue measure $dm = \frac{1}{2|\partial \calD|}drds$ on $M$. We are interested in the Lyapunov exponents of billiard map $F$. By Oseledets's theorem, we know that the Lyapunov exponents exist at $m$-almost every point $x\in M$.

\begin{theorem} [\cite{cb} Theorem 3.1]
Let $M$ be a 2-dimensional compact Riemannian manifold and $F: M\to M$ a $C^2$ diffeomorphism preserving a Borel probability measure $m$ on $M$. Suppose that
\begin{equation}
\int_M\log^+\|DF(x)\|m(dx)< \infty \text{ and } \int_M\log^+\|DF^{-1}(x)\|m(dx)< \infty,
\end{equation}
where $\log^+ = \max\{\log,0\}$. Then there exists an $F$-invariant set $H\subset M$ of full measure, on which all iterations of $F$ are defined on $H$ such that for each $x\in H$ there is a $DF$-invariant decomposition of the tangent space:
\begin{equation}
T_xM = E_1(x)\oplus\cdots\oplus E_k(x)
\end{equation}
for some $k=k(x)$, such that for each non-zero vector $v\in E_i(x)$ the following limit exists:
\begin{equation}
\lim_{n\to \infty}\frac{1}{n}\log\|DF^n(x)v\| = \lambda_i(x)
\end{equation}
where $\lambda_1(x) > \cdots > \lambda_k(x).$
\end{theorem}

 The Lyapunov exponents tell us how nearby trajectories will be separated from each other as the system evolves in time. In our case, $k(x)$ is either 1 or 2. Let $\lambda_1(x) \geq \lambda_2(x)$ be its Lyapunov exponents then by Lemma 3.9 in \cite{cb}: $\lambda_1(x) + \lambda_2(x) = 0$. A point $x\in M$ is hyperbolic if $\lambda_1(x) >0$:  nearby trajectories are separated exponentially fast in the future. There are many billiards in which the Lyapunov exponents are zero. For example, the Lyapunov exponents for any circular billiard are 0 at all points: the trajectories are separated at most linearly. Other similar examples are elliptic billiards. In each of these billiard models, the systems are completely integrable and their collision spaces are foliated by invariant curves.

In what follows, we are going to add some noise each time there is a collision. In the deterministic setting, a point $x$ is mapped to $F(x)$. With the noise added, now the image of $x$ could be in some open neighbourhood of $F(x)$. In this way, a point can escape a region with slow or no expansion even if it needs many iterations depending on the added noise is.


Fix an $\epsilon > 0$. We denote by $B_{\varepsilon}(x)$ the ball of radius $\varepsilon$ and centred at a point $x\in \bR^2$.
Consider a probability measure $\nu_{\varepsilon}$ on $\bR^2$ such that $d\nu_{\varepsilon} =\rho d\bm$ for some measurable density function $\rho$, here $\bm$ is the Lebesgue measure on $\bR^2$. Suppose that the support of $\rho$ contains $B_{\varepsilon}(0,0)$.

Recall that each cylinder $M_i$ is obtained by taking the rectangle $R_i$ and identifying the two vertical edges. The system is randomly perturbed as follows: take a point $x = (r,s) \in M_i$, then perturb this point to another point in $M_i$ with a distribution law given by:
\begin{equation}\label{perturbdef}
\begin{aligned}
    \eta_{\varepsilon}^x(B) &= \nu_{\varepsilon}\left\{u\in \bR^2: u + x \in B \mod{\bZ\partial R_i}\right\}\\
    &=C \int_{B}  \sum_{v\in \bZ\partial R_i}\rho(u-x - v) \bm(du)
\end{aligned}
\end{equation}
for any measurable set $B\subset M_i$. The set $\bZ\partial R_i$ consists of vectors $v$ such that $\frac{1}{n}v \in \partial R_i$ for some $n\in \bZ$; the constant $C$ is the normalising constant. 


\begin{example}
Let $\rho = \frac{1}{\pi\varepsilon^2}$ on $B_{\varepsilon}(0,0)$ and  = 0 elsewhere. A point $x$ goes to $F(x)$ and then jump randomly to a point in a disc of radius $\varepsilon$ centred at $F(x)$ following the distribution $\eta_{\varepsilon}^{F(x)}$. If any part of the disc lies above or below $M_i$ then this part will cut and translated back to $M_i$ by a constant in $\bZ\partial R_i$.
\end{example}

\begin{figure}[h]

\centering
\includegraphics[width=0.5\textwidth]{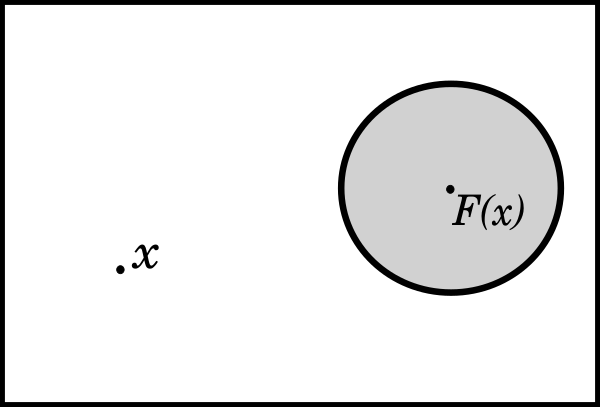}
\caption{Random perturbation when $F(x)$ is far from the boundary}
\end{figure}

\begin{figure}[h]

\centering
\includegraphics[width=0.5\textwidth]{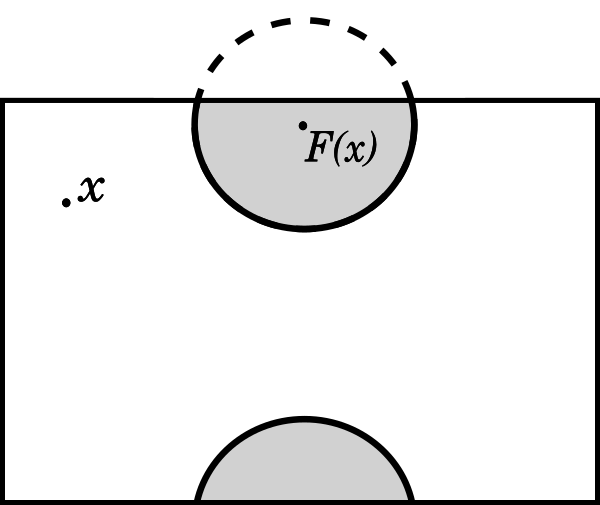}
\caption{Random perturbation when $F(x)$ is near the boundary}
\end{figure}



\begin{definition}\label{randombilmap}
Given a vector $u \in \bR^2$,  we define a function $F_{u}: M \to M$ by:
\begin{equation}
F_{u}(x) = F(x)+ u.
\end{equation}
The function $F_{u}$ is called the perturbed billiard map with $u$.
\end{definition}

For any sequence of vectors $\uu = (u_0, u_1, \dots)$, the compositions of perturbed map with noise given by $\uu$ is defined by:
\begin{equation}
F^n_{\uu} = F_{u_{n-1}}\circ F_{u_{n-2}} \circ \cdot \circ F_{u_0}
\end{equation}
for any $n\geq 1$.

Let $\Upsilon = (\bR^2)^{\bN}$ be the space of all sequences of vectors in $\bR^2$, equipped with the probability measure $\nu_{\varepsilon}^{\bN}$. This is our sample space for the noise. Let $U_n$ be the coordinate mapping:
\begin{equation}
U_n(\uu) = U_n(u_0,u_1,\dots) = u_n
\end{equation}
for any $\uu = (u_0, u_1,\dots)\in \Upsilon$ and $n\geq 0$.

 It is easy to see that the process $\uU = (U_n)_{n \geq 0}$ is a sequence of independent and identically distributed random variables taking values in $\bR^2$ with distribution given by the probability measure $\nu_{\varepsilon}$.
 
 We consider a sequence of random variables $(X_n)_{n\geq 0}$ defined on $\Upsilon$ with values in $M$ such that:
 \begin{itemize}
 \item $X_0$ is a random variable with values in $M$ and some distribution $\mu_0$,
 \item for each $n\geq 1$, $X_n$ is defined by the recurrence relation:
\begin{equation}
	X_n(\uu) = F(X_{n-1}(\uu)) + U_{n-1}(\uu).
	\label{markovchain}
\end{equation}
\end{itemize}
\begin{remark}
If $\varepsilon = 0$ and there is no perturbation, then $\nu_{\varepsilon} = \delta_{(0,0)}$ and $\nu^{\bN}_{\varepsilon}$ is the Dirac measure at the sequence $\underline{0}$ of zero vectors. In this case, the recurrence relation (\ref{markovchain}) is the deterministic billiard map and we obtain a trajectory in the phase space given by the unperturbed billiard map, starting from some point $X_0(\underline{0}) \in M$.
\end{remark}

Let $\uu = (u_0, u_1, \dots ) \in \Upsilon$ be a realisation of the process $\uU$. Let $x_n = X_n(\uu)$. We have:
 \begin{itemize}
 \item $X_0(\uu) = x_0$ is some point in $M$
 \item for each $n\geq 1$ we have:
\begin{equation}
\begin{aligned}
	x_n &= F(x_{n-1}) + u_{n-1} = F_{u_{n-1}}(x_{n-1})\\
	&=F^n_{\uu}(x_0)
	\end{aligned}
\end{equation}
\end{itemize}
 
One can check using the definition of the perturbation in (\ref{perturbdef}) that if we define 
$P: M\times \cB \to [0,1]$ as
\begin{equation}\label{markovkernel}
P(x,B) = \eta_{\varepsilon}^{F(x)}(B)
\end{equation}
for $m$-almost every $x\in M$ and $B\in \cB$.
Then $P$ is a Markov transition function on $M$.

\begin{lemma}\label{muinv}
For any vector $u \in \bR^2$, the map $F_u$ preserves the measure $m$.
Moreover, $m$ is an invariant measure with respect to the Markov transition function $P$.
\end{lemma}
\begin{proof}
The map $F_u$ is the composite of the billiard map $F$ with the translation by $u$. The measure $m$ is $F$-invariant and also translation-invariant, therefore it is $F_u$-invariant. Moreover, $m$ is an invariant measure with respect to the Markov transition function $P$.
For every $B\in \cB$ we have:
\begin{align*}
\cL m(B) &= \int_M P(x,B)m(dx)\\
        &= \int_M \nu_{\varepsilon}(u: F(x) + u \in B) m(dx)\\
        &= \int_M\int_{\bR^2}1_{\{F_{u}(x) \in B\}} \nu_{\varepsilon}(du)m(dx)\\
        &= \int_{\bR^2}\int_M 1_{\{F_{u}(x) \in B\}} m(dx) \nu_{\varepsilon}(du)\\
        &= \int_{\bR^2}m(B)\nu_{\varepsilon}(du)\\
        &= m(B).
\end{align*}
Therefore we have $\cL m = m$.
\end{proof}

Next we state and prove one of the main results for this paper. 
\begin{theorem}\label{ergodicity}
Let $\calD$ be a classical billiard such that the table's boundary satisfies Assumption $\mathbf{(HB)}$. Consider a random perturbation to the system as described in  (\ref{perturbdef}). Then the resulting random billiard is ergodic.
\end{theorem}


\begin{proof}
 By the definition of the process $(X_n)_{n\geq 0}$ in (\ref{markovchain}), we know that 
the process $(X_n)$ defined by equation (\ref{markovchain}) is a Markov process with Markov kernel: $$P(x,B) = \eta_{\varepsilon}^{F(x)} (B)$$ for $x\in M$ and $B\in \cB$. We need to prove that the measure $m$ on $M$ is ergodic for this Markov process. 

By Lemma \ref{muinv}, the measure $m$ is an invariant probability measure for this Markov process. We will now show that this is in fact the unique invariant measure.


Since the transition probability has a density function, we observe that if $B\in \cB$ is such that $m(B) = 0$ then $P(x,B) = 0$ as functions of $x$. Thus the support of every invariant measure for $\cL$ is also has positive measure with respect to the measure $m$. Thus there can be at most countably many of invariant measures for $\cL$.

Recall that for any $x\in M$, the density function of the transition probability $P(x, dy)$ is positive on $B_{\varepsilon}(F(x))$. This condition means that if we start from $x$ then in next step of the process we are allowed to go to anywhere in a ball of radius $\varepsilon$ centred at the point $F(x)$.  Under this condition, there can be almost countably many ergodic invariant measures with respect to $P$. 

Two nearby points are in the same ergodic components due to the perturbation. Because two distinct ergodic measures are mutually singular, the Lebesgue measure must be the only ergodic measure. In fact, this implies that it is the only invariant measure with respect to $P$.


\end{proof}

As we know that ergodicity may not imply hyperbolicity. One such example is an  independent, identically-distributed stochastic process $\{\xi_n\}$. By the law of large numbers, the process $\{\xi_n\}$ is ergodic, but its Lyapunov exponents are zeros. Below we will provide sufficient conditions which guarantees the hyperbolicity for our Markov process $\{X_n\}$.
\begin{theorem}\label{hyperbolicity}
Let $\calD$ be a classical billiard such that the table's boundary satisfies Assumption $\mathbf{(HB)}$. Consider a random perturbation to the system given by the transition function $P$ as defined in \ref{markovkernel}.
Assume that the derivative $DF: M\to \SL(2,\bR)$ satisfies one of the two hypotheses $\mathbf{(H2)}$ and $\mathbf{(H3)}$:

$\mathbf{(H2)}$: there exist non-empty open sets $U$ and $V$ in $M$ such that $DF$ has distinct real eigenvalues on $U$ and only complex eigenvalues on $V$.

$\mathbf{(H3)}$: there exist non-empty open sets $U$ and $V$ in $M$ such that $DF$ has distinct real eigenvalues on $U$ and $V$ but the eigenvectors on $U$ are different from those on $V$.

Let $(X_n)_{n\geq 0}$ be the Markov process with transition $P$ and initial distribution $m$. Let $\lambda_1 \geq \lambda_2$ be the Lyapunov exponents associated to the Markov process $(X_n)_{n\geq 0}$ and the derivative map $DF$. Then  $\lambda_1 > 0$.
\end{theorem}

\begin{proof}
  Suppose to the contrary that $\lambda_1 = 0$. $A = DF$ satisfies the condition $\mathbf{(H1)}$ in Lemma \ref{lambda1=0} as shown in Lemma 3.6 of \cite{cb}. Therefore there exists a measurable family $\xi: x\mapsto \xi_x$ of probability measures on $\bbP$ indexed by $M$ such that for $m$-almost every $x\in M$:
\begin{equation}
	\xi_y = (DF(x))_*\xi_x
\end{equation}
for $P(x, .)$-almost every $y$.

  For any $x \in M$ the support of $P(x,.)$ contains a ball $B_{\varepsilon}(F(x))$ of radius $\varepsilon>0$ and centred at $F(x)$. Consider a partition of $M$ by squares of size $\varepsilon/2$. Since $F$ is an invertible map, $\xi$ is constant on the union of any 4 adjacent squares and hence $m$-almost everywhere on $M$. So there exists a subset $S\subset M$ with $m(S) = 1$ such that $\xi$ is constant at every point in $S$. From now on, we will also use $\xi$ to denote the measure $\xi_x$ of any $x\in S$. 

Let $x \in S$, we have that $\xi = (DF(x))_*\xi$. By iterating the matrix $DF(x)$ we have $$\xi = (DF(x))_*^{n}\xi$$ for every $n\geq 1$. 

Suppose that $DF(x_1)$ has distinct real eigenvalues $\alpha_1$ and $\alpha_2$ for some $x_1\in S$. As $\det{DF(x_1)} = 1$, we can assume that $|\alpha_1|> 1 >  |\alpha_2|$. Let $v_i \in \bR^2$ be a unit eigenvector for $\alpha_i$, $i = 1,2$. Let $v\in \bR^2$ be any nonzero vector non-proportional to $v_2$ and consider the sequence of unit vectors:
\[
  u_n =\frac{DF(x_1)^n(v)}{\norm{DF(x_1)^n(v)}}.
\]

As $n\to \infty$, the angle between $u_n$ and $v_1$ converges to 0 or $\pi$. Therefore the probability measure $\xi$ must be concentrated only on the direction of $v_1$ and $v_2$. In other words, we must have that $\xi  = c_1\delta_{[v_1]} + c_2\delta_{[v_2]}$ for some constants $c_1$ and $c_2$.

Let $x_2\in S$ such that  $DF(x_2)$ has complex eigenvalues. By a change of coordinates, $DF(x_2)$ becomes a rotation matrix. We could assume that $DF(x_2)$'s rotational angle is an irrational multiple of $2\pi$ as the rotational angle varies continuously wherever the billiard map is $C^2$. If the rotational angle is an irrational multiple of $2\pi$ then the probability measure $\xi$ must be the Lebesgue measure on $\bbP$. This is a contradiction to the fact that $\xi$ is discrete.

Let $x_3 \in S$ such that $DF(x_3)$ has distinct real eigenvalues and eigenvectors $w_1$ and $w_2$, such that $\{[w_1], [w_2]\} \cap \{[v_1], [v_2]\} = \emptyset$. Then $\xi = d_1\delta_{[w_1]} + d_2\delta_{[w_2]}$ for some constants $d_1$ and $d_2$. But this contradicts the fact that $\xi  = c_1\delta_{[v_1]} + c_2\delta_{[v_2]}$.



\end{proof}

\subsection{Random non-circular elliptic billiards}
\begin{theorem}
Let $\calD$ be any non-circular elliptic billiard table. Consider any random perturbation to the billiard map on $\calD$ as defined in (\ref{perturbdef}). The resulting random billiard is ergodic and hyperbolic.
\end{theorem}

\begin{proof}
Let  $F:M\to M$ be the billiard map with coordinates $r$ and $s = \sin(\phi)$ as usual. The derivative map $DF$ in the coordinates $r$ and $s= \sin(\varphi)$ is given by

\[ DF(x) = 
	\begin{pmatrix}
		\frac{-1}{\cos(\varphi_1)} & 0\\
		0 &-1\\
	\end{pmatrix}
\begin{pmatrix}
-\tau K + \cos(\varphi) & \tau\\
\tau KK_1 - K\cos(\varphi_1) - K_1\cos(\varphi)&-\tau K_1 +\cos(\varphi_1)
\end{pmatrix}
	\begin{pmatrix}
		1 & 0\\
		0 & \frac{1}{\cos(\varphi)}
	\end{pmatrix}.
\]

We have $\det(DF(x)) = 1$ and \begin{equation}\trace(DF(x)) = \frac{\tau K - \cos(\varphi)}{\cos(\varphi_1)} + \frac{\tau K_1 - \cos(\varphi_1)}{\cos(\varphi)}.\end{equation}

At the point $x_1 \in M$ corresponding to $u = a, v = 0$, $\varphi = \varphi_1 = 0$, we have $\tau = 2a$, $K = K_1 = \frac{a}{b^2}$ and therefore $\trace(DF(x_1)) = \frac{4a^2}{b^2}-2 > 2.$

At the point $x_2 \in M$ corresponding to $u = 0, v = b$, $\varphi = \varphi_1 = 0$, we have $\tau = 2b$, $K = K_1 = \frac{b}{a^2}$ and therefore $\trace(DF(x_2)) = \frac{4b^2}{a^2}-2$. It's clear that $|\trace(DF(x_2))| < 2.$

Since the derivative is a smooth function, $\trace(DF(x)) > 2$ for any $x$ sufficiently close to $x_1$ and similarly $|\trace(DF(x))| < 2$ for $x$ sufficiently close to $x_2$. Thus the perturbed elliptical billiards satisfy the condition $(\mathbf{H2})$ in Theorem \ref{hyperbolicity}. Among those $x$ such that $DF(x)$ has complex eigenvalues, there is a subset of them with positive measure such that at those points the derivative corresponds to irrational rotations. Because of this mixture of real and complex eigenvalues, there cannot be any probability measure on $\bbP$ that is invariant under $DF(x)$ for $m$-almost every $x$.

\end{proof}

\subsection{Lemon billiards}
 Lemon-shaped billiards are a type of convex billiards where the tables are formed by the intersection of two disks \cite{heto}, \cite{cmzz}.

\begin{theorem}
Let $\calD$ be a lemon billiard table. Consider any random perturbation to the billiard map on $\calD$ as defined in (\ref{perturbdef}). The resulting random billiard is ergodic and hyperbolic.
\end{theorem}

\begin{proof}
 The boundary $\partial \calD$ of the table consists of two circular arcs $\calC_1$ and $\calC_2$ of radii $R_1$ and $R_2$ respectively. Without loss of generality, we can assume that $R_1 \leq R_2$. The movement of a particle on a lemon billiard is of the following types: sliding on one of the two circular arcs, or bounce from one arc to the other. 
 
 The derivative map $DF$ at a point $x = (r,\sin(\varphi))$ that corresponds to a sliding movement on $\calC_i$, $i = 1, 2$, is:
\begin{align}\displaystyle
	DF(x) &= \begin{pmatrix}
		1 & \frac{-2R_i}{\sqrt{1-s^2}}\\
		0 & 1\\
	\end{pmatrix}
	= \begin{pmatrix}
		1 & \frac{-2R_i}{\cos(\varphi)}\\
		0 & 1
	\end{pmatrix}\\
	&=\begin{pmatrix}
		1&0\\
		0&\cos(\varphi)\\
	\end{pmatrix}
	\begin{pmatrix}
		1 & -2R_i\\
		0 & 1\\
	\end{pmatrix}
	\begin{pmatrix}
		1&0\\
		0&\frac{1}{\cos(\varphi)}\\
	\end{pmatrix}.
\end{align}
Thus at any points $x$ that corresponds to a sliding on either of the two circular arcs, the derivative has only one eigenvector $\lambda = 1$. The eigenspace for $\lambda = 1$ is generated by $v(x) = \begin{pmatrix}
1\\
0
\end{pmatrix}$. There exists an open subset  $U\subset M$ such that every point in $U$ corresponds to a sliding on the billiard table.\\

Suppose now that $x=(r, \sin(\varphi))$ and $F(x) = (r_1,\sin(\varphi_1))$ such that $x$ and $F(x)$ are based on $\calC_1$ and $\calC_2$ respectively. The derivative map $DF$ at the point $x$ is:

\begin{equation}
\begin{aligned}
DF(x) = 
	\begin{pmatrix}
		1 & 0\\
		0 & \cos(\varphi_1)\\
	\end{pmatrix}
\frac{-1}{\cos(\varphi_1)}\begin{pmatrix}
-\tau K + \cos(\varphi) & \tau\\
\tau KK_1 - K\cos(\varphi_1) - K_1\cos(\varphi)&-\tau K_1 +\cos(\varphi_1)
\end{pmatrix}
	\begin{pmatrix}
		1 & 0\\
		0 & \frac{1}{\cos(\varphi)}
	\end{pmatrix}.
\end{aligned}
\end{equation}

Recall that we have $\det(DF(x)) = 1$ and
\begin{align*}
\trace(DF(x)) &= \frac{\tau K - \cos(\varphi)}{\cos(\varphi_1)} + \frac{\tau K_1 - \cos(\varphi_1)}{\cos(\varphi)}.
\end{align*}
Also note that $K = \frac{1}{R_1}$ and $K_1 = \frac{1}{R_2}$.

In particular, let $x_0$ be the periodic point of period 2 based on $\calC_1$ such that $F(x_0)$ is based on $\calC_2$. It correspond to a perpendicular bounce between the 2 arcs. Let $\tau_0$ be the free path from $x_0$. The derivative of the billiard map at $x_0$ is:

\begin{equation}
\begin{aligned}
DF(x_0) = 
- \begin{pmatrix}
-\tau_0 K +1 & \tau_0\\
\tau_0 KK_1 - K - K_1&-\tau_0 K_1 +1
\end{pmatrix}.
\end{aligned}
\end{equation}

And the trace is:
\begin{align*}
\trace(DF(x_0)) &= \tau_0 (K +K_1) - 2 = \tau_0 \left(\frac{1}{R_1} + \frac{1}{R_2}\right) - 2.\\
\end{align*}

Since $R_1 \leq R_2$ and $\tau_0 < 2R_1$, we have that:

\begin{align*}
\trace(DF(x_0)) &= \tau_0 \left(\frac{1}{R_1} + \frac{1}{R_2}\right) - 2\\
&\leq \tau_0 \left(\frac{2}{R_1} \right) - 2 < 2.\\
\end{align*}


Thus $|\trace(DF(x_0))| < 2$ and $DF(x_0)$ has two distinct complex eigenvalues. In fact, $DF$ has complex eigenvalues on an open neighbourhood of $V\subset M$ $x_0$ as the billiard map is smooth. Thus the perturbed lemon billiards satisfy the condition $(\mathbf{H2})$ of the Theorem \ref{hyperbolicity} and there can be no probability measures on $\bP^1$ that is invariant under the actions of the derivatives $DF(x)$. We have shown that the perturbed lemon billiards are ergodic and hyperbolic.


\end{proof}

\subsection{Random circular billiards}
\begin{theorem}
Let $\calD$ be any circular billiard table of radius $R>0$. Consider any random perturbation to the billiard map on $\calD$ as defined in (\ref{perturbdef}). The resulting random billiard is ergodic, but the Lyapunov exponents are always 0.
\end{theorem}
\begin{proof}
The derivative of the billiard map in this case is:
\begin{align}
	DF(x) &= \begin{pmatrix}
		1 & \frac{-2R}{\sqrt{1-s^2}}\\
		0 & 1\\
	\end{pmatrix}
	= \begin{pmatrix}
		1 & \frac{-2R}{\cos(\varphi)}\\
		0 & 1
	\end{pmatrix}\\
	&=\begin{pmatrix}
		1&0\\
		0&\cos(\varphi)\\
	\end{pmatrix}
	\begin{pmatrix}
		1 & -2R\\
		0 & 1\\
	\end{pmatrix}
	\begin{pmatrix}
		1&0\\
		0&\frac{1}{\cos(\varphi)}\\
	\end{pmatrix}.
\end{align}
Let $(X_n)_{n\geq 0}$ be the corresponding Markov process with transition $P$ and initial distribution $m$. By Theorem \ref{ergodicity}, the dynamical system $(\Omega, \calA, \bP_m, \theta)$ associated to $(X_n)_{n\geq 0}$ is ergodic, and thus the largest Lyapunov exponent in this case is:
\begin{align}\label{lyacos}
	\lambda_1 &= \lim_{n\to \infty}\frac{1}{n}\log \norm{
	\begin{pmatrix}
		1 & -2R\left(\frac{1}{\cos(\varphi_0)} + \cdots + \frac{1}{\cos(\varphi_{n-1})}\right)\\
		0 & 1\\
	\end{pmatrix}}\\
	&=\lim_{n\to \infty} \frac{1}{n}\log \left(2R\left( \frac{1}{\cos(\varphi_0)} + \cdots + \frac{1}{\cos(\varphi_{n-1})}\right)\right).
\end{align}
for $\bP_{m}$-almost every sequence $\omega = (x_0, x_1, \dots)$.
In the above equality, we used the max norm for the matrices.

Recall that any point $x\in M$ has two coordinates $r$ and $s = \sin(\varphi)$. Let $g: M \to [1,\infty)$ be the function defined by:
\[
  g(r,s) = \frac{1}{\cos(\varphi)},
\]
where $s = \sin(\varphi)$.
We have that:
\begin{align}
  \int_M g(x)m(dx) &= \frac{1}{2|\partial \calD|}\int_M \frac{1}{\cos(\varphi)}dsdr\\
  &=\frac{1}{2|\partial \calD|}\int_{-\pi/2}^{\pi/2}d\varphi\int_{\partial\cD} dr\\
  &= \frac{\pi}{2}.
\end{align}
Therefore, by the Birkhoff's Ergodic Theorem, we have:
\begin{equation}\label{sumcos}
  \lim_{n\to \infty}\frac{1}{n}\left( \frac{1}{\cos(\varphi_0)} + \frac{1}{\cos(\varphi_1)} +  \cdots + \frac{1}{\cos(\varphi_{n-1})} \right)  = \frac{\pi}{2}.
\end{equation}
for $\bP_{m}$-almost every sequence $\omega = (x_0, x_1, \dots)$.

Let $\omega = (x_0, x_1, \dots)$ be a sequence such that both Eq. (\ref{lyacos}) and (\ref{sumcos}) hold for $\omega$. Then there exists an integer $N > 0$ large enough such that for every $n\geq N$ we have:
\[
\frac{\pi}{3} \leq \frac{1}{n}\left( \frac{1}{\cos(\varphi_0)} + \frac{1}{\cos(\varphi_1)} +  \cdots + \frac{1}{\cos(\varphi_{n-1})}\right) \leq \pi.
\]
Therefore:
\begin{equation}
\lambda_1 \leq \lim_{n\to \infty} \frac{1}{n}(\log(2R) + \log(n\pi)) = 0.
\end{equation}
This implies that for  $0 \leq \lambda_1 \leq 0$, therefore $\lambda_1 = 0$. 

\end{proof}


\end{document}